\documentclass[12pt,reqno]{amsart}
\numberwithin{equation}{section}
\usepackage{amsmath, amssymb,latexsym,amsrefs,enumerate}
\newtheorem{lem}{Lemma}[section]
\newtheorem{thm}[lem]{Theorem}
\newtheorem{coro}[lem]{Corollary}
\theoremstyle{definition}
\newtheorem{example}[lem]{Example}
\newtheorem{remark}[lem]{Remark} 
\newtheorem{defn}[lem]{Definition}

\newcommand{\boldR}{{\mathbb{R}}}

\newcommand{\field}[1]{\mathbb{#1}}

\newcommand{\R}{\field{R}}
\newcommand{\Z}{\field{Z}}

\renewcommand{\AA}{{\mathcal A}}

\newcommand{\PP}{{\mathcal P}}
\renewcommand{\AA}{{\mathcal A}}
\newcommand{\isdef}{\stackrel{\text{\tiny def}}{=}} 

\newcommand{\cM}{{\mathcal M}}
\newcommand{\cMR}{{\mathcal M}_R}

\title{Bernstein-Szeg\H{o} measures, Banach algebras, and scattering theory}

\author[J.~Geronimo]{Jeffrey~S.~Geronimo}
\address{JSG, School of Mathematics, Georgia Institute of Technology,
Atlanta, GA 30332--0160, USA and Tao Aoqing  Visiting Professor, Jilin University, Changchun, China }
\email{geronimo@math.gatech.edu}
\thanks{JSG is partially supported by Simons Foundation Grant \#210169. JSG would like to thank the JLU-GT Joint Institute for Theoretical Sciences at Jilin University where he was a visiting professor for its hospitality.}

\author[P.~Iliev]{Plamen~Iliev}
\address{PI, School of Mathematics, Georgia Institute of Technology,
Atlanta, GA 30332--0160, USA}
\email{iliev@math.gatech.edu}
\thanks{PI is partially supported by Simons Foundation Grant \#280940.}


\begin{document} 
\date{September 4, 2015}

\begin{abstract}
We give a simple and explicit description of the Bernstein-Szeg\H{o} type 
measures associated with Jacobi matrices which differ from the Jacobi matrix 
of the Chebyshev measure in finitely many entries. We also introduce a class of measures $\cM$ which parametrizes the Jacobi matrices with exponential decay and for each element in $\cM$ we define a scattering 
function. Using Banach algebras associated with increasing Beurling weights, we  
prove that the exponential decay of the coefficients in a Jacobi matrix is completely determined by the decay of the negative 
Fourier coefficients of the scattering function. Combining this result with the Bernstein-Szeg\H{o} type measures we provide different characterizations of the rate of decay of the entries of the Jacobi matrices for measures in $\cM$.
\end{abstract}

\maketitle

\section{Introduction}\label{se1}
Recently there has been interest in measures supported on the unit circle and 
the real line that have analytic weights \cites{do,gm,m,mms,simon,simon1,simon2,ds}. Given a positive Borel measure on a bounded subset of the real line an infinite 
Jacobi matrix can be uniquely associated with it and the problem of interest 
is to study the consequences on the measure when it is assumed that the 
coefficients in the Jacobi matrix tend to their asymptotic values 
exponentially fast. Here it is assumed that the limit of the off diagonal 
elements is not zero. For the nonexponential case this problem was studied in \cites{gc,ger1, ger2, gus}. The exponential case  was considered by the first author in \cite{ger3} and, using the techniques of Banach algebras 
introduced by Baxter \cite{ba}, necessary and sufficient conditions were 
given connecting smoothness properties of the measure with the decay of 
the (recurrence) coefficients in the associated Jacobi matrix. Recent results 
using Riemann-Hilbert techniques, pioneered by Deift and Zhou to solve 
problems in the asymptotics 
of orthogonal polynomials on the unit circle \cites{do,m,mms}, the books by Simon \cites{simon, simon2}, and the works \cites{ds,simon1} have brought renewed interest in this 
problem. In particular, responding to a question of Barry Simon to the first 
author, an extension of Baxter's theory to Banach algebras for analytic 
weights was given in \cite{gm}. In this work the tight connection between the negative Fourier coefficients of the $\mathcal S$ function and the recurrence coefficients was developed and the special role of Bernstein-Szeg\H{o} measures i.e.
 $d\mu=\frac{d\theta}{q(\theta)}$ where $q$ is a positive trigonometric 
polynomial was noted. These types of measures have only a finite number of 
nonzero recurrence coefficients and multiplication of weights on the unit 
circle with exponentially decaying recurrence coefficients by a Bernstein-Szeg\H{o} weight does not change the rate of exponential decay. The techniques in \cite{gm} are for the most part algebraic.

In this paper we introduce a class of measures which parametrizes the Jacobi matrices with exponentially decaying coefficients. Each element in this class gives rise to a scattering function which completely determines the rate of decay of these coefficients. We begin in Section~\ref{se2} with an explicit description of all Bernstein-Szeg\H{o} type measures using just the residue theorem and simple properties of orthogonal polynomials. These measures have been studied by Szeg\H{o} \cite{sz} in the absolutely continuous case
and Geronimus \cite{gr}, Geronimo and Case \cite{gc}, and Damanik and Simon \cite{ds} in the general case. Bernstein-Szeg\H{o} type  measures have an absolutely continuous part analogous to that given above
and perhaps a singular part consisting of a finite number of mass points with masses that have been coined \emph{canonical} in \cite{ds}.  Our results are slightly more detailed than those in the literature since we make explicit the connection between the degree of the trigonometric polynomial generating the  absolutely continuous part of the measure and the minimal index beyond which all  coefficients in the Jacobi matrix attain their asymptotic values. This section is of independent interest since a) all computations are simple and explicit, b) when the moment problem is determinate these measures weakly approximate the starting orthogonality measure, and c) the exponential decay of the recurrence coefficients of analytic absolutely continuous measures is unaffected by multiplication by Bernstein-Szeg\H{o} type weights. In Section~\ref{se3} we 
study the four possible maps between polynomials orthogonal on the unit circle and those on the real line. The various relations between the recurrence coefficients of polynomials orthogonal on the unit circle and those orthogonal on the real line were developed by Killip and Nenciu \cite{kn}. We have included these formulas as well as the explicit relation between the polynomials in order to keep the presentation self-contained. 
In Section~\ref{se4}, following Geronimo \cite{ger3}, we introduce the  Banach 
algebras associated with increasing Beurling weights and relate the decay of the recurrence coefficients to the membership of a certain function (the Jost function) in an appropriate Banach algebra.   
Using the results of Damanik and Simon as a guide, a useful class $\cM$ of analytic measures with positive canonical weights is introduced and the connection with Bernstein-Szeg\H{o} measures is discussed. Then, employing the results and techniques developed in Geronimo and Mart{\'{\i}}nez-Finkelshtein \cite{gm}, we introduce a close relative of the scattering function  $\mathcal{S}$  of Case and Chiu \cite{cc}. As a first result we obtain the theorem of Damanik and Simon \cite{ds} relating the exponential decay of the recurrence coefficients to their asymptotic values and the annulus of meromorphicity for the absolutely continuous part  of the orthogonality measure. The analog of this theorem for polynomials orthogonal on the unit circle can be found in the work of Nevai and Totik~\cite{nt}.  Next, simple necessary and sufficient conditions are given relating the rate of decay of the recurrence coefficients to their asymptotic values and  the rate of decay of the negative Fourier coefficients of the scattering function for measures in
$\cM$. These results are then used to describe different properties of the measures in $\cM$. It is here that the role of Bernstein-Szeg\H{o} measures
is decisive for simplifying the presentation by using  algebraic
arguments to bypass  certain subtleties when the measure has masses.            

\section{Bernstein-Szeg\H{o} weights}\label{se2}
Consider the recurrence formula 
\begin{equation}\label{reconon}
 a_{n+1} p_{n+1}(x)+b_n p_n(x)+ a_n p_{n-1}(x)= x p_n(x),
\end{equation}
with $b_n\in\boldR$ and $a_n>0,\ n>0$. With initial values 
$p_{-1}(x)=0,\ p_0(x)=1$ the Favard theorem says the solution of the
above equation gives a set of polynomials orthonormal with respect to some 
positive measure. Suppose $\lim a_n=a>0, \ \lim b_n= b$ and write two recurrences \cite{gc}, 
\begin{equation}\label{rec2}
\begin{split}
 &p_n(x)=\frac{a}{a_n}[(z-B_{n-1})p_{n-1}(x)+\frac{1}{z}\psi_{n-1}(z)]\\
 &\psi_n(z)=\frac{a}{a_n}\left[\frac{1}{z}\psi_{n-1}(z)+\left(\left(1-\frac{a_n^2}{a^2}\right)z-B_{n-1}\right) p_{n-1}(x)\right],
\end{split}
\end{equation}
where, $B_{n-1}=\frac{b_{n-1}-b}{a}$, $x=a(z+1/z)+b$ and $\psi_0(z)=1$. Here we 
let
$z=\frac{x-b}{2a}-\sqrt{(\frac{x-b}{2a})^2-1}$ where the branch of the square root chosen gives $|z|\le1$ for 
all $x$.
The difference between the above two equations yields
\begin{equation}\label{rec3}
\psi_n(z)=p_n(x)-\frac{a_n}{a}z p_{n-1}(x).
\end{equation}
It will be convenient to write the above two equations as the system
\begin{equation}\label{sys}
\Phi_n(z)=T(n,z)\Phi_{n-1}(z),
\end{equation}
where $\Phi_n(z)=\begin{pmatrix}p_n(x)\\\psi_n(z)\end{pmatrix}$, $\Phi_0(z)=\begin{pmatrix}1\\1\end{pmatrix}$, and
\begin{equation}\label{sys1}
T(n,z)=\frac{a}{a_n}\begin{pmatrix}z-B_{n-1}&\frac{1}{z}\\(1-\frac{a_n^2}{a^2})z-B_{n-1}&\frac{1}{z}\end{pmatrix}.
\end{equation}
Suppose now that $a_{n+1}=a$ and $b_n=b$ for all $n\ge n_0$. The second equation in \eqref{rec2} shows that $\hat\psi_n (z)=\hat\psi_{n_0}(z)$ for $n\geq n_0$, where 
\begin{equation}\label{eq4}
\hat\psi_n (z)=z^n\psi_n(z).
\end{equation}
Consider the first component of \eqref{rec2} with $n$ replaced by $n+1$. If we replace $z$ by $1/z$ and
then subtract the result from the original equation we obtain
\begin{equation}\label{eq11}
p_n(x)=\frac{z^{n+1}\hat\psi_{n_0}(1/z)-(1/z)^{n+1}\hat\psi_{n_0}(z)}
{z-1/z} \qquad \text{for }n\geq n_0.
\end{equation}

We begin by proving a more detailed version of Theorems 1.4 and 1.9 in \cite{ds}.
\begin{thm}\label{bpoly}
Let $\rho$ be a positive Borel measure supported on the real
line with recurrence coefficients $\{a_n\}^\infty_{n=1}$ and 
$\{b_n\}^\infty_{n=0}$, $a_n>0$, $b_n$ real.  Then $a_{n+1}=a$ and
$b_n=b$ for all $n\ge n_0$ if and only if
\begin{equation}\label{measone}
d\rho=\begin{cases}
\sigma (\theta)d x & x=2a\cos\theta+b,\quad 0<\theta<\pi\\
\sum^N_{i=1}\rho_i\delta (x-x_i)d x & x\notin[b-2a,b+2a],
\end{cases}
\end{equation}
with
\begin{equation}\label{meas2}
\sigma(\theta)=\frac{\sin \theta}{a\pi|\hat\psi_{n_0}(z)|^2},\qquad z=e^{i\theta}
\end{equation}
and
\begin{equation}\label{meas3}
\rho_i=\frac{(z_i-1/z_i)}{a\frac d{d x}
(\hat\psi_{n_0}(z)\hat\psi_{n_0}(1/z))}\Big|_{x=x_i}
\end{equation}
where $x_i=a(z_i+1/z_i)+b$, $z_i$ are the zeros of $\hat\psi_{n_0}(z)$ in the unit disk $|z|<1$.
If $a_{n_0}=a$ and $b_{n_0-1}\ne b$, then $\hat\psi _{n_0}$ is a
polynomial of degree $2n_0-1$ in $z$.  If $a_{n_0}\ne a$ then
$\hat\psi_{n_0}$ is a polynomial of degree $2n_0$.
In both cases $\hat \psi_{n_0}(0)>0$, and $\hat\psi_{n_0}(z)$ has
only real simple zeroes for $|z|\leq1$. Except possibly for $z=\pm1$, a zero 
of $\hat\psi_{n_0}(z)$ does not occur at a zero of $\hat\psi_{n_0}(1/z)$. At a 
zero $z_0$ of $\hat\psi_{n_0}$
\begin{equation}\label{decay}
p_n(x_0)=\frac{z_0^{n+1}\hat\psi_{n_0}(1/z_0)}{z_0-1/z_0}
\end{equation}
which shows that for $|z_0|<1$ the sequence $\{p_n(z_0)\}_n$ decays exponentially 
to zero.
\end{thm}
\begin{remark}
The polynomials are square summable at the zeros of $\hat\psi_{n_0}$ for $|z|<1$ and give rise to the eigenvectors  of the Jacobi matrix associated with equation~\eqref{reconon} at its  eigenvalues $x_0=a(z_0+1/z_0)+b$. The zeros of $\hat \psi_{n_0}$ 
for $|z|>1$ are called resonances. 
\end{remark}

\begin{remark}
By abuse of notation, we will often write the function $\sigma(\theta)$ in \eqref{meas2} as a function of $z$, simply as $\sigma(z)$. If we write 
\begin{equation}\label{sigex}
\sigma(z)=\frac{z-1/z}{2a\pi i \hat\psi_{n_0}(z) \hat\psi_{n_0}(1/z)},
\end{equation}
then equation~\eqref{meas3} can be written as

\begin{equation}\label{newmas}
\rho_i=\frac{2\pi i}{\frac{d}{dx} \sigma(z)^{-1}}\Bigg|_{x=x_i}.
\end{equation}
Such masses are said to be canonical (see Definition~\ref{mclass},\ equation~\eqref{rhoj}).
\end{remark}
If $\Psi_n$ and $\Phi_n$ are any two solutions of equation~\eqref{sys} it is 
easy to see that
\begin{equation}\label{wr}
(\Psi_n)^T\begin{pmatrix} 0&-1\\1&0\end{pmatrix}\Phi_n,
\end{equation}
is independent of $n$. Thus, if we 
let $\Psi_n=\Phi^1_n=\begin{pmatrix}p_n^1\\\psi^1_n\end{pmatrix}$ where $p^1_n$ 
are the orthogonal polynomials of the second kind (i.e $p^1_0=0, \ p^1_1(x)=\frac{1}{a_1}$ so $\psi^1_1(z)=\frac{1}{a_1}$), we find that $\psi^1_n(z)p_n(x)-\psi_n(z)p^1_n(z)=\frac{1}{a_1}(p_1(x)-\psi_1(z))=\frac{z}{a}$ where equation~\eqref{rec3} has been used to obtain the last equality.  At a zero $z_0$ of $\psi_n$ we find $\hat\psi^1_n(z_0)=z_0^n\psi^1_n(z_0)=\frac{z_0}{az_0^{-n}p_n(z_0)}=\frac{z_0-1/z_0}{a\hat\psi_n(1/z_0)}$ which gives another representation of $\rho_i$,
\begin{equation}\label{alterep}
\rho_i=\frac{\hat\psi^1_{n_0}(z)}{\frac{d}{dx}\hat \psi_{n_0}(z)}\Bigg|_{x=x_i}.
\end{equation}

\begin{proof} Without loss of generality we can take $a=1/2$ and $b=0$. If $a_{n+1}=1/2$ 
and $b_n=0$ then the second component
of \eqref{rec2} gives
\begin{equation}\label{onehalf}
\psi_{n} (z)=\frac{1}{2a_{n}}((1-4a^2_n)z-2b_{n-1})p_{n-1}(x)+\frac1{2a_{n}}\frac{\psi_{n-1}(z)}z
\end{equation}
so that $\hat\psi_n (z)=\hat\psi_{n_0}(z)$ for $n\geq n_0$.  
From the initial condition $\psi_0=1=p_0$ and
the fact that $z^n p_{n-1}(x)$ is a polynomial in $z$ of exact
degree $2n-1$ it follows by induction  that
$\hat\psi_n(z)$ is a polynomial of degree $2n-1$ if $a_n=1/2$ and
$b_{n-1}\ne 0$ and degree $2n$ if $a_n\ne1/2$.
The fact that $\hat\psi$ has real coefficients shows that for $|z|=1$, $\hat\psi_{n_0}(1/z)=\overline{\hat\psi_{n_0}(z)}$. Thus  the interlacing properties of the zeros of orthogonal polynomials \cite{sz} and \eqref{eq11} show that $\hat\psi_{n_0}$ has no zeros for $|z|=1$ except perhaps at $z=1$, or $z=-1$ and if these zeros are present they must be simple. Also a zero of $\hat\psi_{n_0}$ cannot be a zero of $\psi_{n_0}(1/z)$. At the zeros of $\hat\psi_{n_0}(z)$, $p_n$ has the form given by equation~\eqref{decay}.  To see that the zeros of $\hat\psi$ for $|z|<1$ are real and simple let $p_n^1$ and $p_n^2$ be any two solution of the three term recurrence formula. Then standard manipulations yield,
\begin{align}\label{CD}
a_n(p^1_n(x)p^2_{n-1}(y)-p^1_{n-1}(x)p^2_n(y))&=(x-y)\sum_{i=0}^{n-1}p^1_i(x)p^2_i(y)\\&\nonumber \quad +a_0(p^1_0(x) p^2_{-1}(y)-p^1_{-1}(x)p^2_0(y)).
\end{align}
Setting $p^1=p^2=p$ and $y=\bar x$ then using \eqref{rec3} yields for $n>n_0$,
\begin{align*}
\Im\left(\frac{1}{z}\psi_n(z)\overline{p_n(x)}\right)=(\Im(z)+\Im(1/z))\sum_{i=0}^{n-1}|p_i(x)|^2+\Im(1/z)|p_n(x)|^2,
\end{align*}
which shows the reality of the zeros of $\psi_n$ for $|z|<1$. Writing $x$ and $y$ in terms of $z$ and $w$, then taking the 
limit $w\to z$  in \eqref{CD} and using \eqref{rec3} we find for real $z$,
\begin{align*}
\left(\frac{1}{z}\psi_n(z)\right)'p_n(x)-\frac{1}{z}\psi_n(z)\frac{d}{dz}p_n(x)=(1-\frac{1}{z^2})\sum_{i=0}^{n-1}p_i(x)^2-\frac{1}{z^2}p_n(x)^2,
\end{align*}
which shows that the zeros of $\psi_n(z)$ for $|z|<1$ are simple.
Since $\hat\psi _n(0)=\frac1{2a_n}\,\hat\psi_{n-1}(0)$ we find
$\hat\psi_n(0)=\prod^n_{i=1}\frac1{2a_i}>0$. We now show that the 
polynomials $p_n$ are orthonormal with respect to the measure given 
by \eqref{meas2} and \eqref{meas3}. To this end consider the integral
\begin{equation}\label{eq12}
I=\int p_n (x)x^jd\rho (x)\end{equation} for $j<n$,  and substitute 
the expression above for $\rho(x)$ to obtain
\begin{equation}\label{eq13}
I=\int_{-1}^1 p_n(x)x^j\sigma(\theta)dx+\sum^N_{i=1}
\rho_ip_n(x_i)x^j_i.\end{equation}
The simplicity of the zeros of $\hat \psi_{n_0}$ at $z=1$ or $-1$, if 
they exist, shows the above integral is well defined. For $n\ge n_0$ the 
integral in the above expression can be rewritten using \eqref{eq11}
as
$$\int^1_{-1} p_n(x)x^j\sigma(\theta)dx=\frac{1}{2i}\int^\pi_{-\pi}
z^{n+1}\hat\psi_{n_0} (1/z)x^j\sigma(\theta)\, d\theta.$$
Writing $|\hat\psi_{n_0}(z)|^2=\hat\psi_{n_0}(z)\hat\psi_{n_0}(1/z)$  we see 
that the  integrand is analytic except at the zeros of $\hat\psi_{n_0}(z)$ so 
the residue theorem shows,
\begin{equation}\label{resione}
\frac{1}{2i}\int^\pi_{-\pi} z^{n+1}\hat\psi_{n_0}(1/z)x^j\sigma(\theta)\, d\theta = -\sum^N_{i=1}\rho_i p_n (x_i)x^j_i ,
\end{equation} where we have used \eqref{decay}. 
Substituting \eqref{resione} into 
\eqref{eq13} gives that \eqref{eq12} is equal to zero. That the integral 
of $p_n^2$ is equal to 1 follows from the above arguments after 
using equation~\eqref{eq11} to eliminate $p_n$ and utilizing the 
fact that the residue at $z=0$ is equal to one. For $n<n_0$ the result
follows by induction with the aid of the three term recurrence formula. 

To show the other direction assume that $\rho$ has the form indicated. We prove first that for $n\ge n_0$ the orthonormal polynomials $p_n(x)$ associated
with $\rho$ are given by \eqref{eq11}. Since the right-hand side of formula \eqref{eq11} is invariant under the 
transformation $z\to 1/z$ and since $\hat\psi_{n_0}(0)>0$ we see indeed that equation \eqref{eq11} defines a polynomial in $x$ of degree
$n$. By repeating the above arguments it follows that $p_n(x)$ are orthonormal with respect to $\rho$.

Equation \eqref{eq11} implies that the polynomials 
$p_n$ satisfy the recurrence formula
$$\frac{1}{2} \, p_{n+2}(x) + \frac{1}{2}\, p_{n}(x)=x p_{n+1}(x)$$
for $n\ge n_0$ so that $a_{n+1}=1/2$ and $b_{n+1}=0$ for $n\ge n_0$.
Moreover, using \eqref{eq11} we see that 
$$xp_{n_0}(x)-\frac{1}{2}p_{n_0+1}(x)=\frac{z^{n_0}\hat\psi_{n_0}(1/z)-(1/z)^{n_0}\hat\psi_{n_0}(z)}
{2\left(z-1/z\right)},$$
and since the right-hand side is a polynomial in $x$ of degree at most $n_0-1$, we deduce that $b_{n_0}=0$.
\end{proof}
The only if part of above theorem was known to Geronimus \cite{gr} (see also \cite{gc} appendix A). 

The above measures provide a sequence of approximating measures for any orthogonality measure $\rho(x)$. The result below is well known but we include it for convenience.  
Let $\{ a_n, b_{n-1}\}_{n\ge 1}$ be the coefficients of the orthonormal polynomials $p_n(x)$ corresponding to the measure $\rho(x)$ and let us consider the measure $\rho^{n_0}(x)$ associated with the coefficients $\{a_n^{n_0}, b^{n_0}_{n-1}\}_{n\geq 1}$ where
$$a^{n_0}_i=\begin{cases}a_i & i<n_0\\ a_{n_0} & i\ge n_0\end{cases}\qquad\text{ and }\qquad
b^{n_0}_i=\begin{cases}b_i & i<n_0-1\\ b_{n_0} & i\ge n_0-1.\end{cases}$$
\begin{thm}
Suppose that the moment problem is determinate. Then $\rho^{n_0}(x)$
converges weakly to $\rho(x)$ as $n_0\rightarrow\infty$.
\end{thm}
\begin{proof}
Let $\{p_n^{n_0}\}$ be the orthonormal polynomials with respect to $\rho^{n_0}(x)$. 
Since $p_n^{n_0}=p_n$ for $n\le n_0$ we see that the first $2n_0$ moments for $\rho^{n_0}$ and $\rho$ are the same. The result follows since the moment problem is determinate. 
\end{proof}

\section{Szeg\H{o} maps}\label{se3}

We now consider the relation between polynomials orthogonal on the
unit circle and those orthogonal on the real line. If $\mu$ is a positive Borel measure with infinite
support on the unit circle there is a unique sequence of polynomials $\{\phi_n\}_{n\ge0}$, where $\phi_n$ is of exact degree $n$ in
$z=e^{i\theta}$ with positive leading coefficient $k_n$, such that 
$$\int_{-\pi}^{\pi} \phi_n(z)\overline{\phi_j(z)}d\mu=\delta_{n,j}.$$
These polynomials satisfy the following recurrence formula
\begin{equation}\label{recopuc}
\phi_n(z)=\frac{k_n}{k_{n-1}}(z\phi_{n-1}(z)-\alpha_n\overleftarrow\phi_{n-1}(z)), 
\end{equation}
where $\overleftarrow\phi_n(z)=z^n\bar{\phi_n}(1/z)$ and
\begin{equation}\label{alp}
\alpha_n=-\frac{\phi_n(0)}{k_n}.
 \end{equation}
It follows from the orthonormality of the polynomials that 
\begin{equation}\label{receq}
 1=\frac{k_n^2}{k_{n-1}^2}(1-|\alpha_n|^2).
\end{equation}
We shall call the $\alpha_n$ recurrence coefficients.

\begin{thm}\label{opuc}
Suppose that $w(x)$ is a positive weight function on $(-1,1)$ and let  $\{a_n, b_{n-1}\}_{n\geq 1}$ be 
recurrence coefficients for the corresponding orthonormal polynomials $\{p_n(x)\}$.  For $-\pi<\theta<\pi$ let
\begin{align*}
f_1(\theta)&= \frac{w(\cos\theta)}{|\sin\theta|}\\
f_2(\theta)&= |\sin\theta| w(\cos\theta)\\
f_3(\theta)&= \sqrt{\frac{1+\cos\theta}{1-\cos\theta}}\, w(\cos\theta)\\
f_4(\theta)&= \sqrt{\frac{1-\cos\theta}{1+\cos\theta}}\, w(\cos\theta)\end{align*}
where we assume in each of the four cases above that $w(x)$ is such that 
$f_i$, $i=1,2,3,4$ is a well defined weight function on the unit circle.
Let $\{\phi_n(z)\}$ be the orthonormal polynomials on the unit circle associated with $f_i$ and
$\alpha_n$ the associated recurrence coefficients.  
Then we have the following relations.
For Case~1,
\begin{equation}\label{pn1}
p_n(x)=2\left(1-\frac{\phi_{2n+2}(0)}{k_{2n+2}}\right)^{-1/2}
\left(\frac{z^{-n-1}\phi_{2n+2}(z)-z^{n+1}\phi_{2n+2}(1/z)}{z-1/z}\right)
\end{equation}
$$a^2_n=\frac{1}{4}(1+\alpha_{2n})(1-\alpha^2_{2n+1})(1-\alpha_{2n+2}),$$
and
$$b_n=\frac{1}{2}[(1-\alpha_{2n+2})\alpha_{2n+1}-(1+\alpha_{2n+2})\alpha_{2n+3}].$$
For Case~2,
\begin{equation}\label{pn2}
p_n(x)=\left(1+\frac{\phi_{2n}(0)}{k_{2n}}\right)^{-1/2}
\left(z^{-n}\phi_{2n}(z)+z^n\phi_{2n}(1/z)\right)
\end{equation}
$$a^2_n=\frac{1}{4}(1-\alpha_{2n-2})(1-\alpha^2_{2n-1})(1+\alpha_{2n}),$$
and
$$
b_n=\frac{1}{2}[(1-\alpha_{2n})\alpha_{2n+1}-(1+\alpha_{2n})\alpha_{2n-1}].$$
For Case~3,
\begin{equation}\label{pn3}
p_n(x)=\sqrt{\frac2{1-\frac{\phi_{2n+1}(0)}{k_{2n+1}}}}\left(
\frac{z^{-n}\phi_{2n+1}(z)}{z-1} + \frac{z^n\phi_{2n+1}(1/z)}{1/z-1}\right)
\end{equation}
$$a^2_n=\frac{1}{4}(1+\alpha_{2n-1})(1-\alpha^2_{2n})(1-\alpha_{2n+1}),$$
and
$$
b_n=\frac{1}{2}[(1-\alpha_{2n+1})\alpha_{2n}-(1+\alpha_{2n+1})\alpha_{2n+2}].$$
For Case~4,
\begin{equation}\label{pn4}
p_n(x)=\sqrt{\frac2{1+\frac{\phi_{2n+1}(0)}{k_{2n+1}}}}\left(
\frac{z^{-n}\phi_{2n+1}(z)}{z+1} + \frac{z^n\phi_{2n+1}(1/z)}{1/z+1}\right)
\end{equation}
$$a^2_n=\frac{1}{4}(1-\alpha_{2n-1})(1-\alpha^2_{2n})(1+\alpha_{2n+1})$$
and
$$b_n=\frac{1}{2}[(1-\alpha_{2n+1})\alpha_{2n+2}-(1+\alpha_{2n+1})\alpha_{2n}].$$
\end{thm}

The connections between the recurrence coefficients for polynomials orthogonal on the unit circle and those orthogonal
on the real line were studied by Killip and Nenciu \cite{kn}. We include a proof of these formulas and the explicit connection between the polynomials for the convenience of the reader.

\begin{proof} We begin by noting that the right-hand sides of 
equations~\eqref{pn1}--\eqref{pn4} are invariant under the transformation
$z\to 1/z$. The symmetry of $f_i$ with respect to $\theta$ implies that the coefficients in $\phi_n$ are all real. Also the numerators in equations~\eqref{pn1} and \eqref{pn3}  vanish for $z=1$ while the numerators of \eqref{pn1} and \eqref{pn4} vanish for $z=-1$ which shows that each of the above functions is a polynomial in $x$. Writing
$$\phi_n(z)=\sum^n_{i=0}\phi_{n,i}z^i \text{ and }
\hat \phi_n(z)=\sum^n_{i=0}\phi_{n,n-i}z^i,$$
substituting these formulas into \eqref{pn1}--\eqref{pn4} and using
the fact that $\phi_{n,n}>|\phi_{n,0}|$ for all $n$ we see that $p_n(x)$ are
polynomials in $x$ of degree $n$.  We now show that
\begin{equation}\label{orone}
I=\int^1_{-1} p_n(x) x^j w(x)dx=0\qquad 0\le j<n\end{equation}
and
\begin{equation}\label{ortwo}
\int^1_{-1} p^2_n(x)w(x)dx=1.
\end{equation}
Since the techniques are the same for each of \eqref{pn1}--\eqref{pn4}
(and are those given in Szeg\H{o}~\cite{sz}) we shall show the above
relations only for \eqref{pn3}.  Substituting the right-hand side in equation
\eqref{pn3} into \eqref{orone} yields
\begin{align*}
I&=c_n\int^\pi_0 z^{-n}\frac{\phi_{2n+1}(z)}{z-1}\, x^j(1-\cos\theta)f_3(\theta)d\theta\\
&\quad +c_n\int^\pi_0 z^n\frac{\phi_{2n+1}(1/z)}{1/z-1} \, x^j(1-\cos\theta)
f_3(\theta)d\theta,
\end{align*}
where $c_n=\sqrt{2/(1-\phi_{2n+1}(0)/k_{2n+1})}$.
Letting $\theta\to-\theta$ in the second integral on the right-hand side of
the above equation yields
\begin{align*}
I&=c_n\int^\pi_{-\pi} z^{-n} \frac{\phi_{2n+1}(z)}{z-1}\, x^j(1-\cos\theta)
f_3(\theta)d\theta\\
&= -\frac{c_n}{2}\int^\pi_{-\pi} z^{-n}\phi_{2n+1}(z)x^j(1-1/z)f_3(\theta)d\theta.
\end{align*}
Since $x=\frac{z+1/z}2$ we see using the orthogonality of $\phi_{2n+1}(z)$ to
$z^{-k}$ for $k=0,\dots, 2n$ that $I=0$ when $0\le j<n$.  For the orthonormality write
\begin{align*}
\int^1_{-1}p_n(x)^2w(x)dx&=c^2_n\int^\pi_0 \Biggl( z^{-2n}
\frac{\phi^2_{2n+1}(z)}{(z-1)^2}\\
&\qquad  +\frac{z^{2n}\phi_{2n+1}(1/z)^2}{(1/z-1)^2}
+\frac{2\phi_{2n+1}(z)\phi_{2n+1}(1/z)}{(z-1)(1/z-1)}\Biggr)
(1-\cos\theta)f_3(\theta)d\theta.
\end{align*}
This can be rewritten as
$$\frac{c^2_n}{2}\int^\pi_{-\pi}(-z^{-2n-1}\phi^2_{2n+1}(z)+
\phi_{2n+1}(z)\phi_{2n+1}(1/z))f_3 (\theta)d\theta=
\frac{c^2_n}{2} (1-\phi_{2n+1}(0)/k_{2n+1}).$$
This yields the orthonormality relations.
In order to prove the relations among the recurrence coefficients we write
$$p_n(x)=\sum^n_{i=0} p_{n,i}x^i.$$
Then from the three term recurrence formulas we find
\begin{equation}\label{recan}
a_n=\frac{p_{n-1,n-1}}{p_{n,n}}\end{equation}
and
\begin{equation}\label{recbn}
b_n=\frac{p_{n,n-1}}{p_{n,n}}- \frac{p_{n+1,n}}{p_{n+1,n+1}}\,.\end{equation}
Equating the coefficients of $z^{n-1}$ in the recurrence relation \eqref{recopuc} we find
\begin{equation}\label{reca1}
\frac{\phi_{n+1,n}}{\phi_{n+1,n+1}}=\frac{\phi_{n,n-1}}{\phi_{n,n}}
+\alpha_n\alpha_{n+1}\end{equation}
while the coefficient of $z$ yields,
\begin{equation}\label{reca2}
\frac{\phi_{n+1,1}}{\phi_{n+1,n+1}}=-\alpha_{n}-\alpha_{n+1}
\frac{\phi_{n,n-1}}{\phi_{n,n}}.
\end{equation}
Equating coefficients of $z^n+z^{-n}$ in \eqref{pn3} yields
$$
p_{n,n}=2^n\sqrt{2}\phi_{2n+1,2n+1}(1+\alpha_{2n+1})^{1/2},
$$
where we have used equation~\eqref{alp}. Computing $\frac{p_{n-1,n-1}^2}{p_{n,n}^2}$ then using \eqref{receq} gives $a^2_n$ for case 3. To compute $b_n$, find the coefficient of $z^{n-1}+z^{-n+1}$ in \eqref{pn3} and note that because of the symmetry $z\to1/z$ there is no contribution from the term multiplying $x^n$. Thus 
$$p_{n,n-1}=2^{n-1}\left(\frac{2}{1+\alpha_{2n+1}}\right)^{1/2}(\phi_{2n+1,2n}-\phi_{2n+1,1}+\phi_{2n+1,2n+1}-\phi_{2n+1,0})$$ 
so that  
$$\frac{p_{n,n-1}}{p_{n,n}}=\frac{1}{2(1+\alpha_{2n+1})}\left(\frac{\phi_{2n+1,2n}}{\phi_{2n+1,2n+1}}-\frac{\phi_{2n+1,1}}{\phi_{2n+1,2n+1}}+1-\frac{\phi_{2n+1,0}}{\phi_{2n+1,2n+1}}\right)=\frac{1}{2}\left(\alpha_{2n}+\frac{\phi_{2n,2n-1}}{\phi_{2n,2n}}\right)+\frac{1}{2}.
$$
Incrementing $n$ by one then subtracting and using once again \eqref{reca1} yields the result for case~3. The other cases follow in a similar manner.
\end{proof}

\section{Banach Algebras}\label{se4}

A \emph{Beurling weight}  \cite{simon}*{chapter 5} is a two-sided sequence $\nu=\{\nu(n)\}_{-\infty}^{\infty}$
with the properties
\begin{align}
\nu(0)& = 1,\quad \nu(n)\ge 1, \label{1.5}
\\
\nu(n) & =\nu(-n), \label{1.6} \\
\nu(n+m) & \le \nu(n) \nu(m).\label{1.7}
\end{align}
These properties imply the existence of the limit
\begin{equation}\label{1.8}
\lim_{n\to +\infty} \nu(n)^{1/ n}=\inf \nu(n)^{1/n}  =R\geq 1.
\end{equation}
If $R=1$, then $\nu$ is a \emph{strong} Beurling weight and if $\nu(n)\le\nu(n+1)$ for $n\ge0$ then $\nu$ is an \emph{increasing} Beurling weight.

Each Beurling weight $\nu$ has the associated Banach space
$\ell_\nu$ of two-sided sequences $f=\{f_n \}_{n=-\infty}^{+\infty}$
with
\begin{equation*}\label{1.9}
 \|f\|_\nu   \isdef \sum_{n=-\infty}^{\infty} \nu(n)|f_n|<\infty;
\end{equation*}
this norm extends naturally to any one-sided sequence by completing
the latter with zeros.

Banach algebras may be
associated with the Beurling weights \cite{ba} in the following manner.
Let $a,\ b\in \ell_\nu$; then their convolution is given by
$$
(a*b)(n)=\sum_{k=-\infty}^{\infty}a(k)b(n-k),
$$
which is absolutely convergent by \eqref{1.5} and by \eqref{1.7},
\begin{equation}\label{2.1}
\| a*b \|_\nu=\sum_n \nu(n)\sum_k |a(k)b(n-k)|\le \|a\|_\nu
\|b\|_\nu.
\end{equation}
Thus if we consider the space of functions
\begin{equation}\label{defOfA}
{\mathcal A}_\nu  \isdef  \left\{ f(z)=\sum_{k \in \Z} f_k z^{ k }:
\|f\|_\nu  = \sum_k \nu(k)|f_k|<\infty \right\}\,,
\end{equation}
equation \eqref{2.1} shows that it is closed under multiplication
and so forms an algebra. For $\nu \equiv 1$ we obtain the Wiener
algebra $\mathcal A_1$, containing ${\mathcal A}_\nu$ for any
Beurling weight $\nu$.

On the set of all two-sided sequences $\{d(n)\}_{n\in \Z}$ we define
the projectors $\PP_-$ and $\PP_+$:
$$
\left( \PP_- d\right)(n)=\begin{cases} d(n), & \text{if } n \leq 0,
\\ 0 , & \text{if } n > 0 \,,
\end{cases} \quad \text{and} \quad \left( \PP_+ d\right)(n)=\begin{cases} d(n), & \text{if } n \geq 0,
\\ 0 , & \text{if } n < 0 .
\end{cases}
$$

These projectors can be naturally extended to ${\mathcal A}_\nu$, and give rise
to two subalgebras associated with ${\mathcal A}_\nu$: ${\mathcal A}^+_\nu
\isdef \PP_+(\AA_\nu)$ and ${\mathcal A}^-_\nu \isdef  \PP_-(\AA_\nu)$.
In other words, $\AA_\nu^{\pm}$ are the sets of functions $f\in {\mathcal
A}_\nu$ whose Fourier coefficients $f_n$ vanish for $n<0$ or $n>0$,
respectively. It is easy to see because of \eqref{1.8} that if $f\in
{\mathcal A}_\nu$ then $f(z)$ is continuous for $1/R\le |z|\le R$ (which
is the maximal ideal space associated to ${\mathcal A}_\nu$ also called the Gelfand spectrum) and
analytic for $1/R< |z|< R$. Likewise if $f$ is in ${\mathcal A}^+_\nu$
then $f$ is continuous for $|z|\le R$ and analytic for $|z|<R$ (in
which case the series in \eqref{defOfA} is its Maclaurin expansion
convergent at least in this disk), and if $f$ is in ${\mathcal A}^-_\nu$
then $f$ is continuous for $|z|\ge 1/R$ and analytic for $|z|>1/R$. If $f\in{\mathcal A}_{\nu}$ and $f\ne0$ for $1/R\le|z|\le R$ then $1/f$ is also in the algebra \cites{kr,grs}.
For an increasing Beurling weight $\nu$ we define the increasing Beurling weight $\hat\nu$ as $\hat\nu(n)=(|n|+1)\nu(n)$.
We define $(a, b)\in \hat l_{\nu}$ if  $\sum^\infty_{n=1} n\nu(2n)(|1-4a^2_n|+|b_{n-1}|)<\infty$. 
We start with a simple lemma.
\begin{lem}\label{hatnu}
If $g(z)$ is analytic on the unit circle and $g'(z)\in  {\mathcal A}_{\nu}$, then $g(z)\in  {\mathcal A}_{\hat \nu}$.
\end{lem}
\begin{proof}
If we write $g(z)=\sum_{m\in\Z}g_mz^m$, then $g'(z)=\sum_{m\in\Z}mg_mz^{m-1}\in {\mathcal A}_{\nu}$ implies that \\
$\sum_{m\in \Z}|g_m||m|\nu(m-1)<\infty.$ Since $\nu(m)\le\nu(m-1)\nu(1)$ and $|m|+1\leq 2|m|$ for $m\neq 0$ we see that 
$||g||_{\hat \nu}<\infty$, completing the proof.
\end{proof}

In the sequel we will make use of the following lemma of Krein \cite{kr} (see also Lemma 1 in \cite{ger3}).
\begin{lem}\label{kre}
Suppose that $g\in{\mathcal A}_{\nu}^+$ where $\nu$ is an increasing Beurling weight with $\lim_{n\to\infty}\nu(n)^{\frac{1}{n}}=R$, $g(z_0)=0$ and $|z_0|<R$. Then $\frac{g(z)}{z-z_0}\in {\mathcal A}_{\nu}^+$.
\end{lem}

A useful Lemma proved in \cite{gm}*{Lemma 2} is
\begin{lem}\label{lemgm}
Let $h \in {\mathcal A}_1$, and let $\nu$ be an increasing Beurling weight. Assume that $\PP_-(h) \in {\mathcal A}_\nu^-$. Then for
a function $g$
$$
g\in \AA_1^+ \text{ or } g \in \AA_\nu^- \quad \Rightarrow \quad
\PP_-(g h) \in \AA_\nu^-\,.
$$
\end{lem}

Using the above notations, we start with the following theorem proved in \cite{ger3}.
\begin{thm}\label{lnu} Let  $\nu$ be an increasing Beurling weight with $\lim_{n\to\infty}\nu(n)^{\frac{1}{n}}=R$.
If $(a,b)\in \hat l_{\nu}$ then
\begin{equation}\label{eq3}
\hat f_+(z)= 2zf_+(z)=\lim_{n\to \infty} \hat\psi_n (z)\in {\mathcal  A}^+_{\nu}.
\end{equation}
Moreover, there are two solutions  
$\Phi^{+}(z,n)=\begin{pmatrix} p_{+}(z,n)\\ \psi_{+}(z,n)\end{pmatrix}$ and 
$\Phi^{-}(z,n)=\begin{pmatrix} p_{-}(z,n)\\\psi_{-}(z,n)\end{pmatrix}$ 
of \eqref{sys} that satisfy
\begin{equation}\label{eqplus}
 \lim_{n\to\infty}|z^{-n}p_+(z,n)-1|=0=\lim_{n\to\infty}|z^{-n}\psi_+(z,n)|\text{ for } |z|\le1,
\end{equation}
and
\begin{equation}\label{eqminus}
 \lim_{n\to\infty}|z^{n}p_-(z,n)-1|=0=\lim_{n\to\infty}|z^{n}\psi_-(z,n)-(1-z^2)|\text{ for } |z|\ge1,
\end{equation}
with $z^{-n}\Phi^{+}(z,n) \in A^+_{\nu}$ and $z^{n}\Phi^{-}(z,n) \in A^-_{\nu}$. If $f_+$ has zeros inside the unit circle they must be real, simple, and finite in number. If $f_+$ has zeros on the unit circle they must be simple at $z=1$ or $-1$ or both and $\hat f_+/h(z)\in {\mathcal A}^+_{\nu}$ where
\begin{equation}\label{hz}
h(z)=\left\{\begin{matrix}1-z& \text{if}\ \hat f_+(1)=0\\ 1+z& \text{if}\ \hat f_+(-1)=0\\
1-z^2& \text{if}\ \hat f_+(1)=0=\hat f_+(-1).\end{matrix}\right.
\end{equation}
Furthermore in this case
\begin{equation}\label{derivative}
(1-z^2)\frac{d}{dz}(\hat f_+)\in {\mathcal A}^+_{\nu} 
\end{equation}
and 
\begin{equation}\label{polyf}
\Phi_n(x)=\frac{1}{z-1/z} \left[z\hat  f_+(1/z)\Phi_+(z,n)-\frac{1}{z}\hat f_+(z)\Phi_-(z,n)\right],\ \frac{1}{R}\le |z|\le R.
\end{equation}
\end{thm}

\begin{remark}
Most of the statements and formulas in \cite{ger3} use the function $f_+$, while we will work mainly with the function $\hat  f_+$ related to $f_+$ by \eqref{eq3}.
\end{remark}

\begin{proof}
Most of the above was given in \cite{ger3}*{Theorem 1} and we will give a new simple proof of \eqref{derivative}. From equation~\eqref{onehalf} and equation~\eqref{eq3} it follows that
\begin{equation}\label{hatf}
\frac{\hat f_+(z)}{\beta}=1+\sum_{i=0}^{\infty}[(1-4 a_{i+1}^2)z^2-2b_i z]\frac{z^i p_i(x)}{\beta_i},
\end{equation}
where $\beta_0=1$, $\beta_n=\prod_{i=1}^n\frac{1}{2 a_i}$ and  $\beta=\prod_{i=1}^{\infty}\frac{1}{2 a_i}$.
The polynomials $p_i$ satisfy \cites{nevai, gc}
\begin{equation}\label{poly}
\frac{z^n p_n(x)}{\beta_n}=\frac{1-z^{2n+2}}{1-z^2}+\sum_{i=0}^{n-1}\left\{(1-4a_{i+1}^2)z^2(\frac{1-z^{2n-2i-2}}{1-z^2})-2b_i z(\frac{1-z^{2n-2i}}{1-z^2})\right\}\frac{z^i p_i(x)}{\beta_i}.
\end{equation}
Since $||\frac{1-z^{2n+2}}{1-z^2}||_{\nu}<(n+1)\nu(2n+2)$, the discrete Gronwall's inequality gives
$$
\left\| \frac{z^n p_n(x)}{\beta_n}\right\|_{\nu}\le(n+1)\nu(2n+2)\exp\left(\sum_{i=1}^ni\nu(2i)(|1-4a_i^2|+2|b_{i-1}|)\right).
$$
If equation~\eqref{poly} is differentiated with respect to $z$ then multiplied by $1-z^2$ we obtain with $k_n(z)=(1-z^2)\frac{d}{dz}\frac{z^n p_n(x)}{\beta_n}$,
\begin{align*}
k_n(z)=&-(2n+2)z^{2n+1}+2z\frac{1-z^{2n+2}}{1-z^2}\\&+\sum_{i=0}^{n-1}\left\{(1-4a_{i+1}^2)z^2(\frac{1-z^{2n-2i-2}}{1-z^2})-2b_i z(\frac{1-z^{2n-2i}}{1-z^2})\right\}'(1-z^2)\frac{z^i p_i(x)}{\beta_i}\\&+\sum_{i=0}^{n-1}\left\{(1-4a_{i+1}^2)z^2(\frac{1-z^{2n-2i-2}}{1-z^2})-2b_i z(\frac{1-z^{2n-2i}}{1-z^2})\right\}k_i(z).
\end{align*}
With the above bound on $||\frac{z^n p_n(x)}{\beta_n}||_{\nu}$ and using the fact that $\nu$ is increasing we see that the sum of the first two terms in the above equation is bounded above by $c(n+1)\nu(2n+2)$ so another application of the discrete Gronwall's inequality gives $||k_n(z)||_{\nu}\le \hat c (n+1)\nu(2n+2)$. Differentiating equation~\eqref{hatf} with respect to $z$ then multiplying by $1-z^2$ yields
\begin{align*}
\frac{1-z^2}{\beta}\frac{d}{dz}\hat f_+&=\sum_{i=0}^{\infty}[2(1-4 a_{i+1}^2)z-2b_i](1-z^2)\frac{z^i p_i(x)}{\beta_i}\\&+\sum_{i=0}^{\infty}[(1-4 a_{i+1}^2)z^2-2b_i z]k_i(z).
\end{align*}
Taking the norm in ${\mathcal A}_{\nu}$ and using the above bounds gives the result.
\end{proof}
It is not difficult to see from the recurrence formulas and boundary conditions that $p_-(z,n)=p_+(\bar z,n)=p_+(1/z,n)$ for $|z|=1$ and these relations hold in regions of overlapping analyticity.

\begin{defn}\label{mclass} 
Let $r>1$ and let $f(z)$ be a real analytic function in $\{z:|z|<r\}$ which has only simple real zeros in the closed unit disk $\{z:|z|\leq 1\}$ and $f(0)> 0$. If $\{z_j\}_{j=1}^{M}$ are the zeros of $f$ in the open unit disk $\{z:|z|< 1\}$, we define a positive Borel measure $\rho=\rho_f$ on $\R$ by
\begin{subequations}\label{genm}
\begin{equation}\label{genmf}
d\rho(x)=\sigma(x)\chi_{(-1,1)}(x)dx+\sum_{j=1}^{M}\rho_j\delta (x-x_j)dx,
\end{equation}
where $x_j=\frac{1}{2}(z_j+\frac{1}{z_j})$,
\begin{equation}\label{genmt}
\sigma(x)=\frac{2 \sin \theta}{\pi |f(z)|^2}, \text{ with }x=\cos\theta, \quad z=e^{i\theta},
\end{equation}
\end{subequations}
and $\rho_j> 0$ are arbitrary. We will denote by $\cM$ the class of all Borel measures defined as above for some $f$ and  $\{\rho_j\}_j$ and for each such $f$ we define 
\begin{equation}\label{sfunction}
{\mathcal S}(z)=\frac{f(1/z)}{f(z)}.
\end{equation}
\end{defn}
It is easy to see that for every measure $\rho\in\cM$ as in \eqref{genm} we can find $R>1$ such that the following conditions hold:
\begin{itemize}
\item[{(i)}] $f(z)$ analytic for $|z|<R$ and continuous for $|z|\leq R$;
\item[{(ii)}] For every zero $z_j$ of $f$ with $1/R\leq |z_j|<1$ we have 
\begin{equation}\label{rhoj}
\rho_j=\frac{(z_j-1/z_j)^2}{z_j f'(z_j)f(1/z_j)}.
\end{equation}
\end{itemize}
Indeed, if we take $R>1$ sufficiently close to $1$, all zeros of $f$ in the
unit disk will be inside the circle $|z|<1/R$ and \eqref{rhoj} will be
automatically satisfied. In practice, we will be interested in the largest
possible $R>1$ for which (i) and (ii) above hold. We denote by $\cMR$ the
subclass of measures in $\cM$ which satisfy the additional conditions
(i)-(ii) and we call the masses in (ii) positive canonical weights
\cites{ds,simon1}. 
Thus, $\cM=\cup_{R>1}\cM_R$ and clearly, $\cM_{R_1}\supset \cM_{R_2}$ for $R_1<R_2$.

\begin{remark}\label{not}
For $\rho\in\cM$ we will write $\rho=\rho_f$ to indicate that $f$ is the function in Definition~\ref{mclass} (which determines the absolutely continuous part of $\rho$ and the location of the mass points). Moreover, if $\rho=\rho_f\in \cMR$ then the canonical weights will given by \eqref{rhoj}.
\end{remark}

\begin{remark}\label{even}
From equation~\eqref{genmt} it follows that $\frac{\sigma(x)}{\sin\theta}$ is an even function of $\theta$.
\end{remark}

\begin{remark}\label{bsremark}
Note that the Bernstein-Szeg\H{o} measures discussed in Theorem~\ref{bpoly} can be characterized as the measures $\rho_f\in\cM$ satisfying the following two conditions:
\begin{itemize}
\item $f(z)$ is a polynomial;
\item $\rho_f$ has positive canonical weights at all zeros of $f$ inside the unit circle (or, equivalently, $\rho_f\in\cMR$ for all $R>1$).
\end{itemize}
\end{remark}

\begin{remark}\label{remcv}
It is perhaps useful to stress that if we start with an arbitrary function $f(z)$ satisfying the conditions in Definition~\ref{mclass}, then the $\rho_j$'s computed from \eqref{rhoj} are not necessarily positive. Indeed, if we write $f$ as
\begin{equation}\label{cvf}
f(z)=\prod_{j=1}^{M}(z-z_j)g(z),
\end{equation}
where $\{z_j\}$ are all the zeros of $f$ inside the unit circle, then the sign of $\rho_j$ coincides with the sign of $z_j^{M+1}\prod_{k\neq j}(z_j-z_k)g(z_j)g(1/z_j)$. 
\end{remark}

\begin{lem}\label{lemmod}
Let $f(z)$ be a function satisfying the conditions in Definition~\ref{mclass}. Then there exists a polynomial $\tilde g(z)$ with real roots of modulus greater than $1$, such that the canonical weights ${\tilde\rho}_j$ for the function $f\tilde{g}$ are all positive. Moreover, $\tilde{g}$ can be chosen of degree $2s$ or less, where $s$ is the number of the negative $\rho_j$'s for $f$.
\end{lem}
\begin{proof}
If we write $f$ as in \eqref{cvf} and using the notations in Remark~\ref{remcv}, it is enough to show that, for every fixed $j$, we can find a polynomial $\tilde{g}_j(z)$ with real roots of modulus greater than $1$, such that $\tilde{g}_j(z_j)\tilde{g}_j(1/z_j)<0$ and $\tilde{g}_j(z_k)\tilde{g}_j(1/z_k)>0$ for all $k\neq j$. Indeed, if  $\rho_j<0$ (in \eqref{rhoj} for the function $f(z)$), then multiplying by $\tilde{g}_j(z)$ we will change the sign of $\rho_j$, while preserving the signs of all other $\rho_k$'s, thus proving the statement by induction. It is easy to see now that we can take $\tilde{g}_j(z)=(z-\gamma_1)(z-\gamma_2)$, where $\gamma_1$ and $\gamma_2$ are chosen sufficiently close to $1/z_j$ so that the interval $(1/\gamma_1,1/\gamma_2)$ contains only $z_j$ (i.e. $z_k\not
\in(1/\gamma_1,1/\gamma_2)$ for $k\neq j$). 
\end{proof}

\begin{example}
As an illustration, consider $f(z)=(z-z_1)(z-z_2)$ where $0<z_1<z_2<1$. For the canonical weights for $f$ computed from formula \eqref{rhoj} we have $\rho_1<0$ and $\rho_2>0$. From the construction in the proof of Lemma~\ref{lemmod} we see that we can take $\tilde{g}(z)=(z-\gamma_1)(z-\gamma_2)$, where $0<1/\gamma_1<z_1<1/\gamma_2<z_2$. In this specific example we can take also $\tilde{g}(z)=z-\gamma_2$, where $z_1<1/\gamma_2<z_2$.
\end{example}

A lemma that we will make use of later is the following.
\begin{lem}\label{bsmod}
Suppose that $\rho_f\in\cMR$. Then we can factor $f$ as $f(z)=\hat{q}(z)\hat{f}(z)$, where $\hat{q}(z)$ is a polynomial with real zeros having positive canonical weights at all of its zeros inside the unit circle,  $\hat{q}(0)>0$ and $\hat{f}(z)$ is nonzero for $z\in[-1,1]$.  
\end{lem}

\begin{proof}
Let $\frac{1}{R_+}=\inf\{z^+_i,\ \frac{1}{R}\le z^+_i\le 1\}$ and $\frac{1}{R_-}=\inf\{-z^-_i,\ -1\le z^-_i\le -\frac{1}{R}\}$ where $z^{\pm}_i$ are the zeros of $f$ for $\frac{1}{R}\le |z|\le 1$. Note that $f$ has finitely many zeros for $z\in [-R_-, R_+]$. Indeed, if $R_+<R$ then $f$ is analytic on $[1,R_+]$ and so  has only finitely many zeros. If $R_+=R$ then $f(R)\ne0$ by the definition of canonical weight, so $f$ can have only finitely many positive zeros. A similar argument can be used for the negative zeros. Let $z_1,z_2,\dots,z_N$ be the real zeros of $f$ for $z\in[-R_-, R_+]$ repeated according to the corresponding multiplicities. From Remark~\ref{remcv} it follows that at all $z_j$ with $1/R \leq |z_j|<1$, the polynomial $q(z)=\prod_{k=1}^{N}(z-z_k)$ has positive canonical weights (because the sign of the canonical weight at $z_j$ coincides with the sign of the canonical weight for $f$). Now we can use the construction in Lemma~\ref{lemmod} to define a polynomial $\tilde{q}(z)$ with real zeros of modulus greater than $R$, such that $\hat{q}(z)=q(z)\tilde{q}(z)$ has positive canonical weights also at the points $z_j$ where $|z_j|<1/R$.
\end{proof}

\begin{remark}\label{ed}
If $(a,b)\in{\hat l}_{\nu}$ and with the notations in Theorem~\ref{lnu} one can show \cite{ger2} that the polynomials $\{p_n(x)\}_{n=0}^{\infty}$ defined by equation \eqref{reconon} are orthonormal with respect to the measure $\rho_{{\hat f}_+}\in\cMR$, where $R$ is given in \eqref{1.8} and the non-canonical weights $\rho_j$ for the zeros $z_j$ of ${\hat f}_+$ with $|z_j|<1/R$ can be computed by
\begin{equation}\label{eq7}
\rho_j=\frac{2z_j p_+ (z_j,0)}{\hat f'_+(x_j)}.
\end{equation}
Moreover in this case ${\mathcal S}$ in \eqref{sfunction} is closely related to the scattering function of Case and Chiu~\cite{cc}.
\end{remark}

\begin{defn}\label{lclass} 
We define 
$\hat\ell=\cup\{\hat{l}_{\nu}:\nu(n)=R^{|n|},\; R>1\}$.
\end{defn}

In other words, $\hat{\ell}$ is the set of all recurrence coefficients (or equivalently, Jacobi matrices) which decay exponentially.

We now connect the rate of decay of the coefficients for polynomials orthogonal on the unit circle and those on the real line.

\begin{lem}\label{connectrec}
Suppose that $w(x)$ is a positive weight on $(-1,1)$ with coefficients $(a,b)$ and one of the cases 1--4 of Theorem~\ref{opuc} gives a well-defined $f_i$ that is strictly positive on the unit circle. Let $\alpha=\{\alpha_n\}$ be the recurrence coefficients associated with this weight and let $\nu$ be an increasing Beurling weight. Then 
$$
\alpha\in l_{\hat\nu} \quad \Rightarrow \quad (a,b)\in \hat l_{\nu}.
$$
\end{lem}
\begin{proof}
If we have case~1 then from the equations for the relations between the recurrence coefficients we find,  $(n+1)\nu(2n)|1-4a_n^2|\le2(\hat\nu(2n)|\alpha_{2n}|+\hat\nu(2n+1)|\alpha_{2n+1}|+\hat\nu(2n+2)|\alpha_{2n+2}|)$ where
 we have used that  that $|\alpha_n|<1$ for all $n$ and $\hat\nu(n)$ is 
increasing. Also $(n+1)\nu(2n)|b_{n-1}|\le \hat\nu(2n)|\alpha_{2n-1}|+\hat\nu(2n+1)|\alpha_{2n+1}|$. Since $\hat\nu(2n)<\hat\nu(2n-1)\hat\nu(1)$  the result follows for case~1. 
Cases~2, 3 and 4 follow in a similar manner.
\end{proof}

The following result was proved by Damanik and Simon \cite{ds} and we give an alternate demonstration. Later in Theorems~\ref{thmnozeros}--\ref{expdecay}, the theory of
Beurling weights is used to give a more precise equivalence between the
maximal rate of decay of $(a,b)$ and the summability of the Fourier
coefficients of $\mathcal S$ and $\sigma$ on the boundary of the region of meromorphicity. 
\begin{thm} \label{an}
Let $R>1$. For a Borel measure $\rho$ with recurrence coefficients $(a,b)$ the following conditions are equivalent:
\begin{itemize}
\item[{(i)}] $\limsup(|1-4a^2_n|+|b_{n-1}|)^{1/2n}\le\frac{1}{R}$;
\item[{(ii)}] $\rho\in\cap_{r<R}\cM_r$.
\end{itemize}
\end{thm}

\begin{proof}
If (i) holds then $\forall\ r\in\ (1,R)$ we see that $(a,b)\in \hat l_{\nu}$ where $\nu(n)=r^{|n|}$ so the result follows 
from Theorem~\ref{lnu} and Remark~\ref{ed}. 
For the opposite direction, pick $r\in (1,R)$ and let $\rho=\rho_f$. 
Let $d_1(z)=\prod_{j=1}^{M_1} (z-z_j)$ where $\{z_j\}$ are the zeros of $f$ for $\frac{1}{r}\leq |z|\leq 1$ and $d_2(z)=\prod_{j=1}^{M_2} (z-z'_j)$ where $\{z'_j\}$ are the zeros of $f$ for $1<|z|\le r$. Set $d(z)=d_1(z)d_2(z)$. Then $\frac{f(z)}{d(z)}$ is analytic for $|z|\le r$ and therefore $\frac{f}{d}\in{\mathcal A}_{\hat\nu}$ where $\nu(n)=r^{|n|}$. Moreover, since $\frac{f(z)}{d(z)}$ is nonzero for $\frac{1}{r}\leq|z|\leq r$, we see that $\frac{d}{f}\in{\mathcal A}_{\hat\nu}$.
Thus $w_1(\theta)=\frac{\sigma(\theta)|d(e^{i\theta})|^2}{\sin\theta}\in {\mathcal A}_{\hat\nu}$, and $w_1(\theta)$ is nonzero for $\frac{1}{r}\le|z|\le r$. By Corollary 2 in \cite{gm}, the recurrence coefficients $\{\alpha^{1}\}$ associated with $w_1$ as a weight on the unit circle are in $l_{\hat\nu}$. Note that 
$\frac{\sigma(\theta)|d_1(e^{i\theta})|^2}{\sin\theta}=\frac{w_1(\theta)}{|d_2(e^{i\theta})|^2}$ and therefore by Corollary 3 in \cite{gm} the recurrence coefficients $\{\alpha\}$ associated with $\frac{\sigma(\theta)|d_1(e^{i\theta})|^2}{\sin\theta}$ as a weight on the unit circle are also in $l_{\hat\nu}$.  Lemma~\ref{connectrec} now implies that the recurrence coefficients associated with $d\tilde{\rho}=\frac{\sigma(\theta)|d_1(e^{i\theta})|^2}{|h(e^{i\theta})|^2}dx$ are in $\hat l_{\nu}$ where $h(z)$ is given by \eqref{hz} with $\hat f_+=f$. Finally, Theorems 6 and 7 in \cite{ger3} show that we can add the masses back and not change the decay rate of the coefficients so that $(a,b)\in\hat l_{\nu}$ hence $\limsup(|1-4a^2_n|+|b_{n-1}|)^{1/2n}\le\frac{1}{r}$. The result follows by taking the limit $r\to R$.
\end{proof}

An immediate corollary is,
\begin{coro} \label{mr}
For a Borel measure $\rho$ with recurrence coefficients $(a,b)$ the following conditions are equivalent:
\begin{itemize}
\item[{(i)}] $(a,b)\in \hat{\ell}$;
\item[{(ii)}] $\rho\in\cM$.
\end{itemize}
\end{coro}

\begin{remark}\label{equiv}
Based on the equivalence between $\hat{\ell}$ and $\cM$ established in the above corollary, we can identify the function $\hat{f}_+$ in \eqref{eq3} (if we start with the recurrence coefficients) with the function $f$ in Definition \ref{mclass} (when we start with $\rho=\rho_f$). 
\end{remark}

\begin{remark}
The hypothesis in Corollary~\ref{mr} can be weakened somewhat using the results 
of Geronimo and Nevai \cite{ger2} and also Guseinov \cite{gus}. However this would take us out of the class $\cM$.
\end{remark}

We are interested in studying how the rate of convergence of the recurrence coefficients are reflected in the
rate of convergence of the Fourier coefficients of the measure. We begin with

\begin{lem}\label{relafourier}
If $\nu$ is an increasing Beurling weight, $\rho=\rho_f\in{\mathcal M}$  then
\begin{equation}\label{equivalences}
\log\left(\frac{\sigma}{\sin\theta}\right) \in \mathcal A_\nu \quad  \Rightarrow \quad 
 \text{both }\mathcal S \text{ and } \mathcal S^{- 1} \in \mathcal A_\nu \,.
\end{equation}
\end{lem}
\begin{proof}
The  conditions show that $f\in {\mathcal A}^+_{\nu}$ and is nonzero for $|z|\le R$ where $R$ is given by \eqref{1.8}. If we set 
$$c_k=\frac{1}{2\pi}\int_{-\pi}^{\pi} e^{-ik\theta}\log\left(\frac{\pi\sigma(\theta)}{2\sin(\theta)}\right)d\theta$$
then from Remark~\ref{even}  we see that $c_k$ are real numbers.
Therefore $\mathcal S(z)=\exp(\sum_{k=1}^{\infty}(c_k z^k-c_{-k} z^{-k}))=\exp(\sum_{k=1}^{\infty}(c_k z^k-c_{k} z^{-k}))$ from which the result follows.
\end{proof}

The use of Beurling weights allows us to give a complete characterization of the exponential decay of $(a,b)$ in terms of  the Fourier coefficients of $\mathcal S$. 

\begin{thm}\label{thmnozeros}
Let $(a,b)$ be the recurrence coefficients associated with a measure $\rho$, satisfying one of the equivalent conditions in Corollary~\ref{mr}. 
Suppose that $\nu$ is an increasing Beurling weight with $R>1$ and let $z_1,\ldots, z_i$ be the zeros of $\hat f_+$ such that $1/R\le|z_j|\le 1$. 
Then ${\mathcal P}_-(\prod_{j=1}^i\frac{z-z_j}{1/z-z_j}\mathcal S)\in {\mathcal A}^-_{\hat\nu}$ and $\rho$ has canonical weights for $\frac{1}{R}\le |z|<1$ if and only if $(a,b)\in \hat l_{\nu}$. 
\end{thm}

\begin{proof}
The hypotheses of the theorem show that $\hat f_+(z) \in {\mathcal A}^+_{\hat\nu_0}$ where $\nu_0(n)=r_0^{|n|}$ for some $r_0>1$. 
 If $(a,b)\in \hat l_{\nu}$ then from Theorem~\ref{lnu} and Lemmas~\ref{hatnu}-\ref{kre} we see that $\hat f_+(z) \in {\mathcal A}^+_{\hat\nu}$. If we set $d(z)=\prod_{j=1}^{M}(z-z_j')$, where $\{z_j'\}$ are the zeros of $\hat f_+$ for $|z|\leq 1$, then Lemmas~\ref{hatnu}-\ref{kre} show that $\tilde f_+(z)=\frac{\hat f_+(z)}{d(z)}\in{\mathcal A}^+_{\hat\nu}$ so $\frac{1}{\tilde f_+}\in{\mathcal A}_1^+$. From Lemma~\ref{lemgm}, ${\mathcal P}_-(\tilde{\mathcal S})\in{\mathcal A}^-_{\hat\nu}$ where $\tilde{\mathcal S}=\frac{\tilde f_+(1/z)}{\tilde f_+(z)}$. Since for $|z_j'|<1/R$, $\frac{1/z-z_j'}{z-z_j'}\in{\mathcal A}_{\hat\nu}^-$ and, for $h$ defined by \eqref{hz}, we also have $\frac{h(1/z)}{h(z)}\in{\mathcal A}_{\hat\nu}^-$, we find that ${\mathcal P}_-(\prod_{j=1}^i\frac{z-z_j}{1/z-z_j}\mathcal S)\in {\mathcal A}^-_{\hat\nu}$. 
Now we prove the converse. 
Let $d_1(z)=\prod_{j=1}^{i} (z-z_j)$ where $\{z_j\}$ are the zeros of $\hat f_+$ for $\frac{1}{R}\leq |z|\leq 1$. 
Using Theorem 2 in \cite{gm} we see that the recurrence coefficients $\{\alpha\}$ associated with $\frac{|d_1(e^{i\theta})|^2}{|{\hat f}_+(e^{i\theta})|^2}$ as a weight on the unit circle are in $l_{\hat\nu}$. Lemma~\ref{connectrec} implies that the recurrence coefficients associated with the measure $d{\rho_1}=\frac{\sin \theta \,|d_1(e^{i\theta})|^2}{ |{\hat f}_+(e^{i\theta})|^2|h(e^{i\theta})|^2}dx$ on $[-1,1]$ are in $\hat l_{\nu}$, where $h(z)$ is given by \eqref{hz}. As above, Theorems 6 and 7 in \cite{ger3} show that we can add the masses at the zeros of $\hat f_+$ for $|z|<1$ preserving the decay rate of the coefficients and therefore $(a,b)\in \hat l_{\nu}$.

\end{proof}

This leads to 
\begin{thm}\label{expdecaynozeros}
Let $\nu$ be an increasing Beurling weight with $R>1$ and $d\rho=\sigma (\theta)dx$ be positive absolutely continuous measure, where $x=\cos\theta$, $0<\theta<\pi$.  Set
$\frac{\sigma(-\theta)}{\sin(-\theta)} = \frac{\sigma(\theta)}{\sin\theta}$,
$0<\theta<\pi$. Then 
$$(a,b)\in \hat l_{\nu}$$ and 
$\sigma (z)$ has an  extension to 
$1/R\le |z|\le R$ such that $|\frac{\sigma (z)}{z-1/z}|<\infty$ for all $z$ in this region
if and only if
\begin{equation}\label{dervfour}
\log\left(\frac{\sigma(z)}{z-1/z}\right)\in {\mathcal A}_{\hat \nu}.  
\end{equation}
\end{thm}
\begin{proof}
If $(a,b)\in \hat l_{\nu}$ then Theorem~\ref{lnu} shows that $\lim_{n\to\infty}\hat\psi_n(z)=\hat f_+(z)\in {\mathcal A}^+_{\nu}$ and that $\frac{\sigma(\theta)}{\sin\theta}=\frac{2}{\pi |\hat f_+(z)|^2}=\frac{2}{\pi \hat f_+(z) \hat f_+(1/z)}$ since $\hat f_+$ has real coefficients. Since $\rho$ is absolutely continuous, ${\hat f_+(z)}$ is nonzero for $|z|\leq 1$ by Remark~\ref{ed}. Moreover, since $|\frac{\sigma (z)}{z-1/z}|<\infty$ for $1/R \le |z|\le R$, it follows that ${\hat f_+(z)}$ is nonzero for $|z|\leq R$.  
From Theorem~\ref{lnu} and Lemma~\ref{kre} it follows that $\frac{d}{dz}{\hat f_+(z)}\in {\mathcal A}^+_{\nu}$. Applying Lemma~\ref{hatnu} we see that ${\hat f_+(z)}\in {\mathcal A}^+_{\hat\nu}$.  By the Wiener-Levy theorem we deduce that $\log(\hat f_+(z))\in {\mathcal A}^+_{\hat\nu}$ from which 
equation~\eqref{dervfour} follows.

To show the other direction we see that if equation~\eqref{dervfour} holds then $\frac{\sigma(z)}{z-1/z}\in {\mathcal A}_{\hat \nu}$ so Lemma~\ref{relafourier} and Theorem~\ref{thmnozeros} imply that $(a,b)\in \hat l_{\nu}$.
\end{proof}
This illustrates the  connection between the  rate of decay of the recurrence coefficients and the Fourier coefficients of $\frac{\sigma(z)}{z-1/z}$ in the case when $\frac{\sigma(z)}{z-1/z}$ is finite for $1/R\le |z|\le R$. If this is not the case then slight modifications are necessary. We begin with an analog of Theorem~3 in \cite{gm} modified to include measures with masses.
\begin{thm}\label{expdecayzeros}
Let $\nu$ be an increasing Beurling weight with $R>1$. Let $\rho_{f_3}\in\cMR$ and suppose that $f_3(z)=f_1(z)f_2(z)$, where $f_1$ and $f_2$ are analytic for $|z|< R$, continuous for $|z|\leq R$ and $f_2$ is nonzero for $|z|\leq 1$. Let $(a^j,b^j)$ denote the recurrence coefficients of $\rho_{f_j}$ for $j=1,2,3$. If $(a^j,b^j)\in\hat l_{\nu}$ for $j=1,2$ then $(a^3,b^3)\in \hat l_{\nu}$.
\end{thm}
\begin{proof}
For $j=1,2,3$, let $\mathcal{S}_j$ denote the function associated with $f_j$ by formula \eqref{sfunction}. Set $\tilde d(z)=\prod_{j=1}^i(z-z_j)$ where $\{z_j\}_{j=1}^i$ are the zeros of $f_1(z)$ for $1/R\le|z|\le1$. Then 
$ \tilde{\mathcal S}_1=\frac{\tilde d(z)}{\tilde d(1/z)}{\mathcal S}_1$ and ${\mathcal S}_2$ are in ${\mathcal A}_1$. Also 
$$
\tilde{\mathcal S}_3=\frac{\tilde d(z)}{\tilde d(1/z)}{\mathcal S}_3=
\frac{\tilde d(z)}{\tilde d(1/z)}{\mathcal S}_1{\mathcal S}_2=\tilde{\mathcal S}_1{\mathcal S}_2.
$$
 From Theorem~\ref{thmnozeros} it follows
$$(a^1,b^1)\in \hat l_{\nu}\Rightarrow \tilde{\mathcal S}_1^-={\mathcal P}_-(\tilde{\mathcal S}_1)\in {\mathcal A}^-_{\hat\nu}$$ and
$$(a^2,b^2)\in \hat l_{\nu}\Rightarrow{\mathcal S}_2^-={\mathcal P}_-({\mathcal S}_2)\in {\mathcal A}^-_{\hat\nu}.$$
With $\tilde{\mathcal S}^+_1=\tilde{\mathcal S}_1-\tilde{\mathcal S}_1^-\in {\mathcal A}^+_{1}$ and ${\mathcal S}^+_2={\mathcal S}_2-{\mathcal S}_2^-\in {\mathcal A}^+_{1}$ we see 
$$
{\mathcal P}_-(\tilde{\mathcal S}_3)={\mathcal P}_-(\tilde{\mathcal S}_1{\mathcal S}_2)={\mathcal P}_-(\tilde{\mathcal S}^-_1 {\mathcal S}_2^-+{\tilde{\mathcal S}}^+_1 {\mathcal S}^-_2+{\tilde{\mathcal S}}^-_1 {\mathcal S}^+_2)\in {\mathcal A}_{\hat\nu}^-,
$$
where Lemma~\ref{lemgm} has been used to obtain the last inclusion. The result now follows from Theorem~\ref{thmnozeros}. 
\end{proof}
By Theorem~\ref{bpoly}, the coefficients $(a,b)$ for Bernstein-Szeg\H{o} measures belong to $\hat l_{\nu}$ for every $\nu$. We will use this fact to show that multiplication of the original weight by a Bernstein-Szeg\H{o} weight does not affect the rate of decay of the recurrence coefficients.

\begin{thm}\label{sbweightmass}
Let $\nu$ be an increasing Beurling weight with $R>1$.
Suppose that $\rho_{f}\in{\mathcal M}_R$. We can write $f$ as $f(z)=q(z)\tilde{f}(z)$, with $q(z)=c\prod_{j=1}^{M}(z-z_j)$ where $\{z_j\}$ are the zeros of $f$ inside and on the unit circle and $c$ is chosen so that $q(0)>0$. Let $(a,b)$ and $(\tilde a,\tilde b)$ denote the recurrence coefficients of $\rho_f$ and $\rho_{\tilde{f}}$, respectively. If  $(\tilde a,\tilde b)\in \hat l_{\nu}$ then $(a,b)\in\hat l_{\nu}$.
\end{thm}

\begin{proof}
Using the notation in Lemma~\ref{bsmod}, we can write $f(z)=\hat{q}(z)\hat{f}(z)$ and $\tilde{f}(z)=\hat{f}(z)\tilde{q}(z)$, where $\hat q(z)$ is a polynomial having positive canonical weights at all zeros inside the unit circle (i.e. $\rho_{\hat{q}}$ is a Bernstein-Szeg\H{o} measure) and $\tilde{q}(z)$ is a polynomial with real zeros with absolute values in $(1,R)\cup(R,\infty)$. Moreover, the zeros of $\tilde{q}(z)$ for $1<|z|< R$ are zeros of $\tilde f$. Similarly to the proof of Theorem~\ref{expdecaynozeros} we deduce that $\tilde{f}\in {\mathcal A}_{\hat\nu}^{+}$. Using $\tilde{f}(z)=\hat{f}(z)\tilde{q}(z)$, the properties of the zeros of $\tilde q$, the fact that $\frac{1}{z-z_0}\in  {\mathcal A}_{\hat\nu}^{+}$ for $|z_0|>R$ and Lemma~\ref{kre}, we see that $\hat{f}(z)\in {\mathcal A}_{\hat\nu}^{+}$ and $\hat{f}(z)$ is nonzero for $|z|\leq 1$. Thus $\hat{f}(1/z)\in {\mathcal A}_{\hat\nu}^{-}$, $\frac{1}{\hat{f}(z)}\in  {\mathcal A}_{1}^{+}$ and therefore ${\mathcal P}_{-} (\mathcal{S}_{\hat{f}})\in {\mathcal A}_{\hat\nu}^{-}$ by Lemma~\ref{lemgm}. The proof now follows from Theorem~\ref{expdecayzeros}.
\end{proof}

Our final result is a simplification of Theorem 14 in \cite{ger3}.
\begin{thm}\label{expdecay}
Let $\nu$ be an increasing Beurling weight with $R>1$ and $\rho (x)$ be a positive measure with absolutely continuous
part $\sigma (\theta)$, $x=\cos\theta$, $0<\theta<\pi$ and recurrence coefficients $(a,b)$.  Set
$\frac{\sigma(-\theta)}{\sin(-\theta)} = \frac{\sigma(\theta)}{\sin\theta}$,
$0<\theta<\pi$. 

If $\rho\in{\mathcal M}_R$ and there exists a polynomial $d(z)$ with real coefficients so that $\log\left(\frac{\sigma (z)d(z)d(1/z)}{z-1/z}\right)\in {\mathcal A}_{\hat\nu}$ then $(a,b)\in\hat l_{\nu}$.

Conversely, if $(a,b)\in \hat l_{\nu}$ and  $\sigma (z)$ has a meromorphic extension to 
$1/R <z<R$ such that $|\sigma (z)|<\infty$ for $|z|=R$ then $\rho\in{\mathcal M}_R$ and there exists a polynomial $d(z)$ with real coefficients so that $\log\left(\frac{\sigma (z)d(z)d(1/z)}{z-1/z}\right)\in {\mathcal A}_{\hat\nu}$. 
\end{thm}

\begin{proof}
Let $\rho=\rho_f\in\cM_R$ and $\log\left(\frac{\sigma(z)d(z)d(1/z)}{z-1/z}\right)\in {\mathcal A}_{\hat\nu}$. Since $\log\left(\frac{d(z)d(1/z)}{f(z)f(1/z)}\right)\in {\mathcal A}_{\hat\nu}$ we see that $d(z)d(1/z)$ must cancel all zeros of $f(z)$ for $\frac{1}{R}\leq|z|\leq R$. Moreover $\log((z-z_0)(1/z-z_0))\in {\mathcal A}_{\hat\nu}$ for $|z_0|<\frac{1}{R}$ and therefore without any restriction we can assume that $d(z)d(1/z)$ cancels all zeros of $f(z)$ for $|z|\leq R$.  If we set $d_1(z)=\prod_{j=1}^{M}(z-z_j)$ where $\{z_j\}$ are the zeros of $f$ for $|z|\leq 1$, then we must have  $d(z)d(1/z)=d_1(z)d_1(1/z)d_2(z)d_2(1/z)$ for some polynomial $d_2(z)$ with real coefficients, which has no zeros on the unit circle. 
From Theorem~\ref{expdecaynozeros}, the recurrence coefficients $(\hat a, \hat b)$ for $d\hat\rho(x)=\sigma(\theta)|d(e^{i\theta})|^2 dx$ are in $\hat l_{\nu}$. Applying Theorem~\ref{expdecayzeros} we see that the recurrence coefficients $(\tilde a, \tilde b)$ for $\sigma(\theta)|d_1(e^{i\theta})|^2 dx$ are also in $\hat l_{\nu}$. Theorem~\ref{sbweightmass} now completes the proof in this direction.

Conversely, if $(a,b)\in \hat l_{\nu}$ then Theorem~\ref{lnu} and Lemmas~\ref{hatnu}-\ref{kre} show that $\hat f_+(z)\in {\mathcal A}^+_{\hat\nu}$.  
Moreover, if we set $d(z)=\prod_{j=1}^{M}(z-z_j)$ where $\{z_j\}$ are the zeros of $\hat{f}_+(z)$ for $|z|< R$, then the function $\hat{f}_+(z)/d(z)$ belongs to $ {\mathcal A}_{\hat\nu}^{+}$ and is nonzero for $|z|\leq R$. Therefore $\log\left(\frac{d(z)}{\hat{f}_+(z)}\right)\in {\mathcal A}_{\hat\nu}^{+}$ from which the result follows.
\end{proof}

\section*{Acknowledgments}
We would like to thank a referee for helpful comments and suggestions that led to an improved version of this paper.

\end{document}